\documentclass{amsart}
\usepackage{amssymb}
\usepackage{amsmath}
\usepackage{amsfonts}

\newtheorem{theorem}{Theorem}
\theoremstyle{plain}

\newtheorem{example}{Example}

\newtheorem{lemma}{Lemma}

\numberwithin{equation}{section}

\begin{document}
\title[Generalized Kantowski-Sachs type spacetime metrics]{Generalized
Kantowski-Sachs type spacetime metrics and their harmonicity}
\author{Murat ALTUNBA\c{S}}
\address{Erzincan Binali Y\i ld\i r\i m University, Faculty of Science and
Art, Department of Mathematics, 24030, Erzincan-Turkey.}
\email{maltunbas@erzincan.edu.tr}

\begin{abstract}
In this paper, we deal with harmonic metrics with respect to\ generalized
Kantowski-Sachs type spacetime metrics. We also consider the Sasaki,
horizontal and complete lifts of generalized Kantowski-Sachs type spacetime
metrics to tangent bundle and study their harmonicity.

\textbf{AMS Classification 2020: }53C30, 53C43.

\end{abstract}

\maketitle

\section {Introduction}

The Kantowski-Sachs spacetime metric (briefly, K-S metric) is a
metric for an anisotropic, homogeneous space with spatial section of
topology $%
\mathbb{R}
\times S^{2},\ $\cite{Kantowski}. This metric is a particular case of the
metrics 
\begin{equation}
g=dt^{2}-X^{2}(t)dr^{2}-Y^{2}(t)[d\theta ^{2}+f^{2}(\theta )d\phi ^{2}],
\label{1}
\end{equation}%
with respect to the local coordinates $(t,r,\theta ,\phi )$. Moreover, $X(t)$%
, $Y(t)$ and $f(\theta )$ are smooth functions on a smooth manifold $M$. The
metric (\ref{1}) is famous as generalized Kantowski--Sachs spacetime metric
(briefly, GK-S metric) and it is called as Kantowski--Sachs, Bianchi
type-III or type-I if $f(\theta )$ is equal to $\sin \theta ,\ \sinh \theta $
or $\theta $, respectively.

Several geometric properties of GK-S metrics were discussed in a very recent
paper of Shaikh and Chakraborty, \cite{Shaikh}. In this paper, the authors
calculated various curvature tensors and obtained a sufficient condition for
GK--S metric to pose a perfect fluid.

Harmonicity is an interesting topic not only in differential geometry but
also in analysis and theoretical physics. In \cite{Chen}, Chen and Nagano
obtained some conditions for harmonicity of two Riemannian metrics with
respect to each other. On the other hand, Oniciuc investigated the
harmonicity of two well-known metrics, namely, the Sasaki and the
Cheeger-Gromoll metrics, with respect to each other on tangent and unit
tangent bundles, \cite{Oniciuc}. In \cite{Zaeim}, Zaeim and his
collaborators also considered G\"{o}del-type spacetime metrics in order to
designate harmonic metrics with respect to them and studied harmonic G\"{o}%
del-type spacetime metrics lifted to the tangent bundle.

This paper is organized as follows: In Section 2, we give basic facts about
harmonic metrics and tangent bundles. In section 3, we study harmonic
metrics with respect to GK-S type metrics. Section 4 deals with the harmonic
lifted metrics (Sasaki lifts, horizontal lifts and complete lifts) on the
tangent bundle. In this final section, we show that the harmonicity
condition of the base manifold metrics is same with their Sasaki, horizontal
and complete lift metrics to the tangent bundle.

\section{Preliminaries}

\subsection{Harmonic metrics}

Throughout the paper, all geometric objects are considered as smooth. The
Einstein summation rule is used and the range of the indices $i,j,k,h,\cdot
\cdot \cdot $ being $\{1,\cdot \cdot \cdot ,m\},$ unless otherwise stated.$\ 
$Let $(M,g)$ be a $m-$dimensional manifold covered by a system of coordinate
neighborhoods $(U,x^{i}).$ Then a pseudo-Riemannian metric $g$ is locally
expressed as 
\begin{equation*}
g=g_{ij}dx^{i}dx^{j}.\ 
\end{equation*}%
The Christoffel symbols (components) of the Levi-Civita connection $\nabla $
of $g$ are given by 
\begin{equation*}
\Gamma _{ij}^{k}=\frac{1}{2}g^{kl}(\partial _{i}g_{jl}+\partial
_{j}g_{il}-\partial _{l}g_{ij}),\ 
\end{equation*}%
where $(g^{kl})$ denotes the inverse matrix of $(g_{ij})$ and $\partial _{i}=%
\frac{\partial }{\partial x^{i}}$ are the basis vectors on $(U,x^{i}).$ The
curvature tensor of $g$ is locally written as 
\begin{equation*}
R_{ijk}^{h}=\partial _{i}\Gamma _{jk}^{h}-\partial _{j}\Gamma
_{ik}^{h}+\Gamma _{il}^{h}\Gamma _{jk}^{l}-\Gamma _{jl}^{h}\Gamma
_{ik}^{l}.\ 
\end{equation*}

Let $(M,g_{1})$ and $(N,g_{2})$ be two pseudo-Riemannian manifolds of
dimensions $m$ and $n,$ respectively and $f:M\rightarrow N$ be a map. Let $%
U\subset M$ be an open set with coordinates $(x^{1},\cdot \cdot \cdot ,x^{m})
$ and $V\subset N$ be another open set with coordinates $(y^{1},\cdot \cdot
\cdot ,y^{n})$ such that $f(U)\subset V,$ and assume that $f$ is locally
expressed by $y^{\alpha }=f^{\alpha }(x^{1},\cdot \cdot \cdot ,x^{m}),\alpha
=1,\cdot \cdot \cdot,n.$ Then the second fundamental form\textit{\ }of $f$, denoted by $%
\beta (f),$ is given by 
\begin{equation}
\beta (f)(\frac{\partial }{\partial x^{i}},\frac{\partial }{\partial x^{j}}%
)^{\gamma }=\{\frac{\partial ^{2}f^{\gamma }}{\partial x^{i}\partial x^{j}}%
-^{M}\! \Gamma _{ij}^{k}\frac{\partial f^{\gamma }}{\partial x^{k}}+^{N}\!\Gamma
_{\alpha \beta }^{\gamma }\frac{\partial f^{\alpha }}{\partial x^{i}}\frac{%
\partial f^{\beta }}{\partial x^{j}}\}\frac{\partial }{\partial y^{\gamma }},
\label{2}
\end{equation}%
and that of the tension field $\tau (f)$ of $f$ is 
\begin{equation}
\tau (f)=\text{tr}\beta (f)=g^{ij}\{\frac{\partial ^{2}f^{\gamma }}{\partial
x^{i}\partial x^{j}}-^{M}\!\Gamma _{ij}^{k}\frac{\partial f^{\gamma }}{%
\partial x^{k}}+^{N}\!\Gamma _{\alpha \beta }^{\gamma }\frac{\partial
f^{\alpha }}{\partial x^{i}}\frac{\partial f^{\beta }}{\partial x^{j}}\}%
\frac{\partial }{\partial y^{\gamma }},  \label{3}
\end{equation}%
where $^{M}\Gamma _{ij}^{k}$ and $^{N}\Gamma _{\alpha \beta }^{\gamma }$ are
the Christoffel symbols of the Levi-Civita connections of the metrics $g_{1}$
and $g_{2}$ respectively.

The map $f$ is a totally geodesic map if and only if $\beta (f)=0$, and the
map $f$ is said to be harmonic if $\tau (f)=0,$ \cite{Eells}.

For the concept of (relative) harmonic metrics on a pseudo-Riemannian
manifold, Chen and Nagano proved the following theorem in \cite{Chen}.

\begin{theorem}
\cite{Chen} A pseudo-Riemannian metric $d$ on a pseudo-Riemannian manifold $%
(M,g)$ is harmonic with respect to $g$ if the identity map $%
I:(M,g)\rightarrow (M,d)$ is harmonic.
\end{theorem}

From (\ref{3}) we have the following lemma.

\begin{lemma}
A pseudo-Riemannian metric $d$ on $M$ is harmonic with respect to another
pseudo-Riemannian metric $g$ on $M$ if 
\begin{equation}
\text{tr}(g^{-1}(^{d}\Gamma ^{k}-\Gamma ^{k}))=0,  \label{4}
\end{equation}%
where $^{d}\Gamma ^{k}$ and $\Gamma ^{k}$ are Christoffel symbols of $d$ and 
$g$ respectively.
\end{lemma}

We refer to \cite{Chen2} for further details regarding the material of
harmonic maps.

\subsection{Tangent bundle}

We recall the basic information about tangent bundles from \cite{Yano}. Let $%
(M,g)$ be a $m-$dimensional manifold and denote by $\pi :TM\rightarrow M$
its natural projection from tangent bundle $TM$ to $M$. The tangent bundle $%
TM$ of $M$ is a $2m-$dimensional manifold. If $(U,x^{i}),$ $i=1,\cdot \cdot
\cdot ,m$ is a coordinate neighbourhood of $M,$ then $(\pi ^{-1}(U),x^{i},x^{%
\bar{\imath}}=u^{i})$, $\bar{\imath}=m+1,\cdot \cdot \cdot ,2m,$ where $%
(u^{i})$ are the coordinates in each tangent space $T_{q}M$ at $q\in M$ with
respect to the natural base $\{\frac{\partial }{\partial x^{i}}(q)\},\ q$ is
a point in $U$ whose coordinates are $(x^{i}).$

Let $X=X^{i}\frac{\partial }{\partial x^{i}}$ be a vector field in $U$ $%
\subset $ $M$. The vertical lift $X^{V}\ $and the horizontal lift $X^{H}$ of 
$X$ are given, with respect to the induced coordinates, by 
\begin{equation}
X^{V}=X^{i}\partial _{\bar{\imath}}  \label{5}
\end{equation}%
and%
\begin{equation}
X^{H}=X^{i}\partial _{i}-u^{a}\Gamma _{ak}^{i}X^{k}\partial _{\overline{i}},
\label{6}
\end{equation}%
where $\partial _{i}=\frac{\partial }{\partial x^{i}},\ \partial _{\bar{%
\imath}}=\frac{\partial }{\partial u^{i}}$ and $\Gamma _{ij}^{k}$ are the
coefficients of the Levi-Civita connection $\nabla $ of $g.$

In local charts $U\subset M,$ if we write $X_{(i)}=\frac{\partial }{\partial
x^{i}},$ then from (\ref{5}) and (\ref{6}), we get the following local
expressions, respectively%
\begin{eqnarray}
X_{(i)}^{H} &=&\delta _{i}^{h}\partial _{h}-(u^{a}\Gamma _{ai}^{h})\partial
_{\bar{h}},  \label{7} \\
X_{(i)}^{V} &=&\delta _{i}^{h}\partial _{\bar{h}},  \notag
\end{eqnarray}%
with respect to natural frame $\{\partial _{h},\partial _{\bar{h}}\},$ where 
$\delta _{i}^{h}$ denotes the Kronecker delta. We call the set $\{E_{\alpha
}\}=\{E_{i},E_{\bar{\imath}}\}=\{X_{(i)}^{H},X_{(i)}^{V}\}$ the frame
adapted to the Levi-Civita connection $\nabla $ of $g$ in $\pi
^{-1}(U)\subset TM.$

From the equations (\ref{5}), (\ref{6}) and (\ref{7}), we get%
\begin{equation}
X^{V}=X^{i}E_{\bar{\imath}},\ X^{V}=\left( 
\begin{array}{c}
{{}0} \\ 
{{}X}^{i}%
\end{array}%
\right) ,  \label{8}
\end{equation}%
\begin{equation}
X^{H}=X^{i}E_{i},\ X^{H}=\left( 
\begin{array}{c}
{{}X^{i}} \\ 
{{}0}%
\end{array}%
\right) ,  \label{9}
\end{equation}%
with respect to the adapted frame $\{E_{\alpha }\}.$

Now we recall well-known three metrics, namely the Sasaki lift metric, the
horizontal lift metric and the complete lift metric on the tangent bundle.

The Sasaki lift metric $^{S}g$ of the metric $g$ on $TM$ is defined by 
\begin{eqnarray*}
^{S}g\left( {}X^{H},{}Y^{H}\right) &=&{}g(X,Y), \\
^{S}g\left( {}X^{H},Y^{V}\right) &=&0, \\
^{S}g\left( {}X^{V},{}Y^{V}\right) &=&{}g(X,Y),
\end{eqnarray*}%
for all vector fields $X,Y\in \chi (M).$ The matrix representations of the
Sasaki lift metric $^{S}g$ and its inverse $^{S}g^{-1}$ are given by,
respectively%
\begin{equation}
^{S}g=\left( 
\begin{array}{cc}
g_{ij} & 0 \\ 
0 & g_{ij}%
\end{array}%
\right) ,\ \ ^{S}g^{-1}=\left( 
\begin{array}{cc}
g^{ij} & 0 \\ 
0 & g^{ij}%
\end{array}%
\right) ,  \label{10}
\end{equation}%
with respect to the adapted frame $\{E_{\alpha }\}.$

For the Levi-Civita connection coefficients of the Sasaki lift metric $%
^{S}g,\ $we have%
\begin{eqnarray}
^{S}\Gamma _{ij}^{k}\ &=&\ ^{S}\Gamma _{i\overline{j}}^{\overline{k}}=\Gamma
_{ij}^{k},\quad \ ^{S}\Gamma _{i\overline{j}}^{k}=\frac{1}{2}R_{hji}^{k}u^{h},
\label{11} \\
^{S}\Gamma _{\overline{i}j}^{k} &=&\frac{1}{2}R_{hij}^{k}u^{h},\ \quad ^{S}\Gamma
_{ij}^{\overline{k}}=-\frac{1}{2}R_{ijh}^{k}u^{h},  \notag
\end{eqnarray}%
where $R$ is the curvature tensor of $(M,g).$

The horizontal lift metric $^{H}g$ of the metric $g$ on $TM$ is defined by%
\begin{eqnarray*}
^{H}g\left( {}X^{H},{}Y^{H}\right) &=&0, \\
^{H}g\left( {}X^{H},Y^{V}\right) &=&{}g(X,Y), \\
^{H}g\left( {}X^{V},{}Y^{V}\right) &=&{}0,
\end{eqnarray*}%
for all vector fields $X,Y\in \chi (M).$ The matrix representations of the
horizontal lift metric $^{H}g$ and its inverse $^{H}g^{-1}$ are given by,
respectively%
\begin{equation}
^{H}g=\left( 
\begin{array}{cc}
0 & g_{ij} \\ 
g_{ij} & 0%
\end{array}%
\right) ,\ \ ^{H}g^{-1}=\left( 
\begin{array}{cc}
0 & g^{ij} \\ 
g^{ij} & 0%
\end{array}%
\right) ,  \label{12}
\end{equation}%
with respect to the adapted frame $\{E_{\alpha }\}.$

Non-zero components of the Levi-Civita connection of the horizontal lift
metric $^{H}g\ $are given by 
\begin{equation}
^{H}\Gamma _{ij}^{k}\ =\ ^{H}\Gamma _{i\overline{j}}^{k}=\Gamma _{ij}^{k}.
\label{13}
\end{equation}

Finally, we will introduce the complete lift metric $^{C}g$ of the metric $g$
on $TM.\ $In order to do this, we should remind the complete lift $f^{C}$ of
a function $f$ and complete lift $X^{C}\ $of a vector field $X$ on $TM$,
where $f$ and $X$ are defined on $M.\ $

Let $\omega $ be a 1-form on $M.$ The evaluation map is a function such that 
$\imath \omega :TM\rightarrow 
\mathbb{R}
,$ $\imath \omega (p,u)=\omega _{p}(u)$. According to evaluation map, the
complete lift of a function $f$ is defined by $f^{C}=\imath (df).$ The
complete lifts of vector fields are determined by their actions on these
complete lift functions. That is to say, for vector fields $\widetilde{X}\ $%
and $\widetilde{Y}\ $on $TM,$ $\widetilde{X}(f^{C})=\widetilde{Y}(f^{C})$ if
and only if $\widetilde{X}=\widetilde{Y}$, for all functions $f$. For a
vector field $X$ on $M,$ its complete lift $X^{C\text{ }}$is the vector
field on $TM$ defined by $X^{C}(f^{C})=(Xf)^{C}.\ $For a
pseudo-Riemannian manifold $(M,g)$, the tangent bundle $TM$ can be equipped
with the complete lift metric $^{C}g$ as 
\begin{equation}
^{C}g=(X^{C},Y^{C})=(g(X,Y))^{C},  \label{14}
\end{equation}%
for arbitrary vector fields $X,Y\in \chi (M).$ The complete lift metric $%
^{C}g$ is given in the matrix form%
\begin{equation}
^{C}g=\left( 
\begin{array}{cc}
u^{k}\frac{\partial g_{ij}}{\partial x^{k}} & g_{ij} \\ 
g_{ij} & 0%
\end{array}%
\right) ,  \label{15}
\end{equation}%
with respect to the local coordinates $(x^{i},u^{i})$ on $TM.\ $Non-zero
components of the Levi-Civita connection of the complete lift metric $^{C}g\ 
$are given by%
\begin{eqnarray*}
^{C}\Gamma _{ij}^{k}\  &=&\ ^{C}\Gamma _{i\overline{j}}^{\bar{k}}=\
^{C}\Gamma _{\overline{i}j}^{\bar{k}}=\Gamma _{ij}^{k}, \\
\ ^{C}\Gamma _{ij}^{\overline{k}} &=&u^{l}\frac{\partial }{\partial x^{l}}%
\Gamma _{ij}^{k}.
\end{eqnarray*}

\section{Harmonic GK-S type metrics}

To obtain the harmonicity condition of the metrics defined in (\ref{1}), we
express the following theorem.

\begin{theorem}
Let $(M,g)$ be a GK-S spacetime with the metric (\ref{1}) with respect to
the local coordinates $(t,r,\theta ,\phi ).\ $Then, the GK-S type metric 
\begin{equation}
\widehat{g}=dt^{2}-\widehat{X}^{2}(t)dr^{2}-\widehat{Y}^{2}(t)[d\theta ^{2}+%
\widehat{f}^{2}(\theta )d\phi ^{2}]  \label{17}
\end{equation}%
is harmonic with respect to $g$ if and only if%
\begin{equation}
\begin{array}{c}
\frac{1}{X^{2}}(\widehat{X}^{\prime }\widehat{X}-X^{\prime }X)+\frac{1}{Y^{2}%
}(\widehat{Y}^{\prime }\widehat{Y}-Y^{\prime }Y)+\frac{1}{Y^{2}f^{2}}(%
\widehat{Y}^{\prime }\widehat{Y}\widehat{f}^{2}-Y^{\prime }Yf^{2})=0\text{
and} \\ 
-\widehat{f}\widehat{f}^{\prime }+ff^{\prime }=0.%
\end{array}
\label{18}
\end{equation}
\end{theorem}

\begin{proof}
Let $(M,g)$ be a GK-S spacetime with the metric%
\begin{equation*}
g=dt^{2}-X^{2}(t)dr^{2}-Y^{2}(t)[d\theta ^{2}+f^{2}(\theta )d\phi ^{2}]
\end{equation*}%
and $\widehat{g}$ be an arbitrary GK-S type metric which is harmonic with
respect to the metric $g.$ Put $X(t)=X,\ Y(t)=Y\ $and $f(\theta )=f.\ $The
matrix representation of the metric $g$ is 
\begin{equation}
g=\left( 
\begin{array}{cccc}
1 & 0 & 0 & 0 \\ 
0 & -X^{2} & 0 & 0 \\ 
0 & 0 & -Y^{2} & 0 \\ 
0 & 0 & 0 & -Y^{2}f^{2}%
\end{array}%
\right) .  \label{19}
\end{equation}%
Obviously, 
\begin{equation}
g^{-1}=\left( 
\begin{array}{cccc}
1 & 0 & 0 & 0 \\ 
0 & \frac{1}{-X^{2}} & 0 & 0 \\ 
0 & 0 & \frac{1}{-Y^{2}} & 0 \\ 
0 & 0 & 0 & \frac{1}{-Y^{2}f^{2}}%
\end{array}%
\right) .  \label{20}
\end{equation}%
Put $\Gamma ^{k}=(\Gamma _{ij}^{k}),i,j=1,\cdot \cdot \cdot ,4$ and $%
k=1,\cdot \cdot \cdot ,4,$ then%
\begin{equation*}
\Gamma ^{1}=\left( 
\begin{array}{cccc}
0 & 0 & 0 & 0 \\ 
0 & X^{\prime }X & 0 & 0 \\ 
0 & 0 & Y^{\prime }Y & 0 \\ 
0 & 0 & 0 & f^{2}Y^{\prime }Y%
\end{array}%
\right) ,
\end{equation*}%
\begin{equation*}
\Gamma ^{2}=\left( 
\begin{array}{cccc}
0 & \frac{X^{\prime }}{X} & 0 & 0 \\ 
\frac{X^{\prime }}{X} & 0 & 0 & 0 \\ 
0 & 0 & 0 & 0 \\ 
0 & 0 & 0 & 0%
\end{array}%
\right) ,
\end{equation*}%
\begin{equation*}
\Gamma ^{3}=\left( 
\begin{array}{cccc}
0 & 0 & \frac{Y^{\prime }}{Y} & 0 \\ 
0 & 0 & 0 & 0 \\ 
\frac{Y^{\prime }}{Y} & 0 & 0 & 0 \\ 
0 & 0 & 0 & -f^{\prime }f%
\end{array}%
\right) ,
\end{equation*}%
\begin{equation*}
\Gamma ^{4}=\left( 
\begin{array}{cccc}
0 & 0 & 0 & \frac{Y^{\prime }}{Y} \\ 
0 & 0 & 0 & 0 \\ 
0 & 0 & 0 & \frac{f^{\prime }}{f} \\ 
\frac{Y^{\prime }}{Y} & 0 & \frac{f^{\prime }}{f} & 0%
\end{array}%
\right) .
\end{equation*}%
The Levi-Civita components $\widehat{\Gamma }^{1},\cdot \cdot \cdot ,%
\widehat{\Gamma }^{4}$ of the metric $\widehat{g}$ are obtained by replacing 
$\widehat{X},\ \widehat{Y}$ and $\widehat{f}$ instead of $X,Y$ and $f$ in $%
\Gamma ^{1},\cdot \cdot \cdot ,\Gamma ^{4}.$ By direct computations, we get
the only non-zero traces as follows:%
\begin{eqnarray*}
\text{{}tr}(g^{-1}(\widehat{\Gamma }^{1}-\Gamma ^{1})) &=&-(\frac{1}{X^{2}}(%
\widehat{X}^{\prime }\widehat{X}-X^{\prime }X)+\frac{1}{Y^{2}}(\widehat{Y}%
^{\prime }\widehat{Y}-Y^{\prime }Y) \\
&&+\frac{1}{Y^{2}f^{2}}(\widehat{Y}^{\prime }\widehat{Y}\widehat{f}%
^{2}-Y^{\prime }Yf^{2})), \\
\text{tr}((g^{-1}(\widehat{\Gamma }^{3}-\Gamma ^{3})) &=&-\frac{1}{Y^{2}f^{2}%
}(-\widehat{f}\widehat{f}^{\prime }+ff^{\prime }).
\end{eqnarray*}%
So, due to (\ref{4}), the identity map $I:(M,g)\rightarrow (M,\widehat{g})$ is
harmonic if and only if 
\begin{eqnarray*}
\frac{1}{X^{2}}(\widehat{X}^{\prime }\widehat{X}-X^{\prime }X)+\frac{1}{Y^{2}%
}(\widehat{Y}^{\prime }\widehat{Y}-Y^{\prime }Y)+\frac{1}{Y^{2}f^{2}}(%
\widehat{Y}^{\prime }\widehat{Y}\widehat{f}^{2}-Y^{\prime }Yf^{2}) &=&0\ 
\text{and} \\
-\frac{1}{Y^{2}f^{2}}(-\widehat{f}\widehat{f}^{\prime }+ff^{\prime }) &=&0.
\end{eqnarray*}%
Thus the proof is complete.
\end{proof}

\begin{example}
\label{exp1}Let us define two GK-S type spacetime metrics on $M$ as follows:%
\begin{eqnarray*}
\widehat{g}_{1} &=&dt^{2}-c_{1}{}^{2}-c_{2}{}^{2}[d\theta ^{2}+\sinh
^{2}\theta d\phi ^{2}], \\
g_{1} &=&dt^{2}-e_{1}{}^{2}-e_{2}{}^{2}[d\theta ^{2}+\theta ^{2}d\phi ^{2}],
\end{eqnarray*}%
where $c_{1},c_{2},e_{1},e_{2}$ are arbitrary constants with respect to the
local coordinates $(t,r,\theta ,\phi )$. These metrics are special cases of
the Bianchi type-III metric and the Bianchi type-I metric, respectively.
Obviously, the metrics $\widehat{g}_{1}$ and $g_{1}$ satisfy the equation (%
\ref{18})$_{1}$, but (\ref{18})$_{2}\ $is not valid for these metrics. So,
the metric $\widehat{g}_{1}$ is not harmonic with respect to the metric $%
g_{1}.$
\end{example}

\section{Harmonic lifted GK-S type metrics}

This section is devoted to study harmonicity of the Sasaki lift, horizontal
lift and complete lift of the GK-S type metrics to the tangent bundle.

\begin{theorem}
Let $(M,g)$ be a GK-S spacetime with the metric defined in (\ref{1}).$\ $%
Then, the Sasaki lift metric $^{S}\widehat{g}$ is harmonic with respect to
the Sasaki lift metric $^{S}g$ if and only if 
\begin{eqnarray*}
\frac{1}{X^{2}}(\widehat{X}^{\prime }\widehat{X}-X^{\prime }X)+\frac{1}{Y^{2}%
}(\widehat{Y}^{\prime }\widehat{Y}-Y^{\prime }Y)+\frac{1}{Y^{2}f^{2}}(%
\widehat{Y}^{\prime }\widehat{Y}\widehat{f}^{2}-Y^{\prime }Yf^{2}) &=&0\ 
\text{and} \\
-\widehat{f}\widehat{f}^{\prime }+ff^{\prime } &=&0,\ 
\end{eqnarray*}%
where $\widehat{g}$ is given in (\ref{17}).
\end{theorem}

\begin{proof}
Let $(M,g)$ be a GK-S spacetime with the metric given in (\ref{1}) and $%
\widehat{g}$ be an arbitrary metric on $M$ given by (\ref{17}). Set $%
(t,r,\theta ,\phi )=(x^{1},x^{2},x^{3},x^{4})$. From (\ref{4}), the metric $%
^{S}\widehat{g}$ is harmonic with respect to $^{S}g$ if and only if the
following relations hold:%
\begin{equation}
\left\{ 
\begin{array}{l}
\text{tr}(^{S}g^{-1}(^{S}\widehat{\Gamma }^{k}-^{S}\!\Gamma ^{k}))=0, \\ 
\text{tr}(^{S}g^{-1}(^{S}\widehat{\Gamma }^{k}-^{S}\!\Gamma ^{\bar{k}}))=0,%
\end{array}%
\right.   \label{21}
\end{equation}%
for indices $k=1,\cdot \cdot \cdot ,4,$ where $^{S}\widehat{\Gamma }^{k}$
and $^{S}\Gamma ^{k}$ are the Christoffel symbols of the metrics $^{S}%
\widehat{g}$ and $^{S}g,$ respectively. Due to (\ref{10}) and (\ref{11}),
the first equation of (\ref{21}) reduces to 
\begin{equation*}
\text{tr}(g^{ij}(\widehat{\Gamma }_{ij}^{k}-\Gamma _{ij}^{k}))=0,\ \forall
i,j,k=1,\cdot \cdot \cdot ,4.\ 
\end{equation*}%
It is obvious that this equation leads same results (\ref{18}).

Having in mind (\ref{10}) and (\ref{11}), the second equation of (\ref{21})
becomes%
\begin{equation*}
\text{tr}(g^{ij}(\widehat{R}_{ij0}^{k}-R_{ij0}^{k}))=0,\ \forall
i,j,k=1,\cdot \cdot \cdot ,4,
\end{equation*}%
where $R_{ij0}^{k}=R_{ijh}^{k}u^{h}$ and $\widehat{R}_{ij0}^{k}=\widehat{R}%
_{ijh}^{k}u^{h}.$ The non-zero components of the curvature tensor $R$ are
obtained as follows by direct calculations:%
\begin{eqnarray*}
{}R_{120}^{1} &=&u^{2}XX^{^{\prime \prime }},\ R_{120}^{2}=u^{1}\frac{%
X^{^{\prime \prime }}}{X}, \\
R_{130}^{1} &=&u^{3}YY^{^{\prime \prime }},\ R_{130}^{3}=u^{1}\frac{%
Y^{^{\prime \prime }}}{Y}, \\
R_{140}^{1} &=&u^{4}f^{2}YY^{^{\prime \prime }},\ R_{140}^{4}=u^{1}\frac{%
Y^{^{\prime \prime }}}{Y}, \\
R_{230}^{2} &=&-u^{3}\frac{X^{\prime }YY^{\prime }}{X},\ R_{230}^{3}=-u^{2}%
\frac{XX^{\prime }Y^{\prime }}{Y}, \\
R_{240}^{2} &=&u^{4}\frac{f^{2}X^{\prime }YY^{\prime }}{X},\
R_{240}^{4}=-u^{2}\frac{XX^{\prime }Y^{\prime }}{Y}, \\
R_{340}^{3} &=&u^{4}f(f(Y^{\prime })^{2}-f^{\prime \prime }),\
R_{340}^{4}=-u^{3}\frac{f(Y^{\prime })^{2}-f^{\prime \prime }}{f}.
\end{eqnarray*}%
The non-zero components of the curvature tensor $\widehat{R}$ are determined
by replacing $\hat{X},\ \widehat{Y}$ and $\widehat{f}$ by $X,Y$ and $f$ in
above relations. Direct computations show that the left-hand side of the
second equation of (\ref{21}) vanishes identically. So, the metric $^{S}%
\widehat{g}$ is harmonic with respect to $^{S}g$ if and only if%
\begin{eqnarray*}
\frac{1}{X^{2}}(\widehat{X}^{\prime }\widehat{X}-X^{\prime }X)+\frac{1}{Y^{2}%
}(\widehat{Y}^{\prime }\widehat{Y}-Y^{\prime }Y)+\frac{1}{Y^{2}f^{2}}(%
\widehat{Y}^{\prime }\widehat{Y}\widehat{f}^{2}-Y^{\prime }Yf^{2}) &=&0\ 
\text{and} \\
-\widehat{f}\widehat{f}^{\prime }+ff^{\prime } &=&0.
\end{eqnarray*}%
Thus the proof is complete.
\end{proof}

\begin{theorem}
Let $(M,g)$ be a GK-S spacetime with the metric defined in (\ref{1}).$\ $%
Then, the horizontal lift metric $^{H}\widehat{g}$ is harmonic with respect
to the horizontal lift metric $^{H}g$ if and only if 
\begin{eqnarray*}
\frac{1}{X^{2}}(\widehat{X}^{\prime }\widehat{X}-X^{\prime }X)+\frac{1}{Y^{2}%
}(\widehat{Y}^{\prime }\widehat{Y}-Y^{\prime }Y)+\frac{1}{Y^{2}f^{2}}(%
\widehat{Y}^{\prime }\widehat{Y}\widehat{f}^{2}-Y^{\prime }Yf^{2}) &=&0\text{
and} \\
-\widehat{f}\widehat{f}^{\prime }+ff^{\prime } &=&0,
\end{eqnarray*}%
where $\widehat{g}$ is given in (\ref{17}).
\end{theorem}

\begin{proof}
According to (\ref{4}), to find the harmonicity condition of $^{H}\widehat{g}
$ with respect to $^{H}g$, we deal with the following relations 
\begin{equation}
\left\{ 
\begin{array}{l}
\text{tr}(^{H}g^{-1}(^{H}\widehat{\Gamma }^{k}-^{H}\!\Gamma ^{k}))=0, \\ 
\text{tr}(^{H}g^{-1}(^{H}\widehat{\Gamma }^{k}-^{H}\!\Gamma ^{\bar{k}}))=0,%
\end{array}%
\right.   \label{22}
\end{equation}%
for all indices $k=1,\cdot \cdot \cdot ,4,$ where $^{H}\widehat{\Gamma }^{k}$
and $^{H}\Gamma ^{k}$ are the components of the Levi-Civita connection of
the metrics $^{H}\widehat{g}$ and $^{H}g,$ respectively. Using the equations (%
\ref{12}) and (\ref{13}), we have 
\begin{equation*}
\text{tr}(g^{-1}(\widehat{\Gamma}^{k}-\Gamma ^{k}))=0,\ k=1,\cdot \cdot \cdot ,4.
\end{equation*}%
This equation demonstrates that $^{H}\widehat{g}$ is harmonic with respect
to $^{H}g$ if and only if the equation (\ref{18}) holds. Thus the proof of
this theorem is complete.
\end{proof}

\begin{theorem}
Let $(M,g)$ be a GK-S spacetime with the metric defined in (\ref{1}).$\ $In
this case, the complete lift metric $^{C}\widehat{g}$ is harmonic with
respect to the complete lift metric $^{C}g$ if and only if the following
equations are satisfied:%
\begin{eqnarray}
\frac{1}{X^{2}}(\widehat{X}^{\prime }\widehat{X}-X^{\prime }X)+\frac{1}{Y^{2}%
}(\widehat{Y}^{\prime }\widehat{Y}-Y^{\prime }Y)+\frac{1}{Y^{2}f^{2}}(%
\widehat{Y}^{\prime }\widehat{Y}\widehat{f}^{2}-Y^{\prime }Yf^{2}) &=&0\text{
and}  \notag \\
-\widehat{f}\widehat{f}^{\prime }+ff^{\prime } &=&0,  \label{222}
\end{eqnarray}%
where $\widehat{g}$ is given in (\ref{17}).
\end{theorem}

\begin{proof}
If $(M,g)$ be a GK-S spacetime with the metric defined in (\ref{1}), then
from (\ref{15}), for the complete lift metric of the metric $g$ with respect
to the local coordinates $(x^{i},u^{i}),\ i=1,\cdot \cdot \cdot ,4,$ we have%
\begin{equation}
\tiny
^{C}{}g=\left( 
\begin{array}{cccccccc}
0 & 0 & 0 & 0 & 1 & 0 & 0 & 0 \\ 
0 & -2u^{1}XX^{\prime } & 0 & 0 & 0 & -X^{2} & 0 & 0 \\ 
0 & 0 & -2u^{1}YY^{\prime } & 0 & 0 & 0 & -Y^{2} & 0 \\ 
0 & 0 & 0 & -(2u^{1}YY^{\prime }f^{2}+2u^{3}Y^{2}ff^{\prime }) & 0 & 0 & 0 & 
-Y^{2}f^{2} \\ 
1 & 0 & 0 & 0 & 0 & 0 & 0 & 0 \\ 
0 & -X^{2} & 0 & 0 & 0 & 0 & 0 & 0 \\ 
0 & 0 & -Y^{2} & 0 & 0 & 0 & 0 & 0 \\ 
0 & 0 & 0 & -Y^{2}f^{2} & 0 & 0 & 0 & 0%
\end{array}%
\right) ,  \label{23}
\end{equation}%
\begin{equation}
\tiny
^{C}g{}^{-1}=\left( 
\begin{array}{cccccccc}
0 & 0 & 0 & 0 & 1 & 0 & 0 & 0 \\ 
0 & 0 & 0 & 0 & 0 & -\frac{1}{X^{2}} & 0 & 0 \\ 
0 & 0 & 0 & 0 & 0 & 0 & -\frac{1}{Y^{2}} & 0 \\ 
0 & 0 & 0 & 0 & 0 & 0 & 0 & -\frac{1}{Y^{2}f^{2}} \\ 
1 & 0 & 0 & 0 & 0 & 0 & 0 & 0 \\ 
0 & -\frac{1}{X^{2}} & 0 & 0 & 0 & \frac{2u^{1}X^{\prime }}{X^{3}} & 0 & 0
\\ 
0 & 0 & -\frac{1}{Y^{2}} & 0 & 0 & 0 & \frac{2u^{1}Y^{\prime }}{Y^{3}} & 0
\\ 
0 & 0 & 0 & -\frac{1}{Y^{2}f^{2}} & 0 & 0 & 0 & \frac{2(u^{1}fY^{\prime
}+u^{3}f^{\prime }Y)}{Y^{3}f^{3}}%
\end{array}%
\right) .  \label{24}
\end{equation}%
The non-zero components of the Christoffel symbols of the metric $^{C}g$
are obtained by direct calculations as follows:%
\begin{eqnarray}
^{C}\Gamma _{22}^{1} &=&XX^{\prime },\quad ^{C}\Gamma _{33}^{1}=Y^{\prime }Y,\quad ^{C}\Gamma
_{44}^{1}=YY^{\prime }f^{2},  \label{255} \\
^{C}\Gamma _{12}^{2} &=&\frac{X^{\prime }}{X},  \notag \\
^{C}\Gamma _{13}^{3} &=&\frac{Y^{\prime }}{Y},\quad ^{C}\Gamma _{44}^{3}=-ff^{\prime }, 
\notag \\
^{C}\Gamma _{14}^{4} &=&\frac{Y^{\prime }}{Y},\quad ^{C}\Gamma _{34}^{4}=\frac{f^{\prime }%
}{f},  \notag \\
^{C}\Gamma _{22}^{\bar{1}} &=&u^{1}(XX^{\prime \prime }+(X^{\prime
})^{2}),\quad^{C}\Gamma _{33}^{\bar{1}}=u^{1}(YY^{\prime \prime }+(Y^{\prime })^{2}),
\notag \\
^{C}\Gamma _{44}^{\bar{1}} &=&f(u^{1}f(Y^{\prime })^{2}+2u^{3}YY^{\prime
}f^{\prime }+u^{1}YfY^{\prime \prime }),  \notag \\
^{C}\Gamma _{12}^{\overline{2}} &=&\frac{u^{1}(XX^{\prime \prime }+(X^{\prime
})^{2})}{X^{2}},\quad ^{C}\Gamma _{\overline{1}2}^{\overline{2}}=^{C}\Gamma _{\overline{2}%
1}^{\overline{2}}=\frac{X^{\prime }}{X},  \notag
\end{eqnarray}%
\begin{eqnarray}
^{C}\Gamma _{13}^{\overline{3}} &=&\frac{u^{1}(YY^{\prime \prime }-(Y^{\prime
})^{2})}{Y^{2}},\quad^{C}\Gamma _{\overline{1}3}^{\bar{3}}=^{C}\Gamma _{\overline{3}1}^{%
\overline{3}}=\frac{Y^{\prime }}{Y},  \notag \\
^{C}\Gamma _{44}^{\overline{3}} &=&-u^{3}((f^{\prime })^{2}+ff^{\prime \prime
}),\quad ^{C}\Gamma _{\overline{4}4}^{\overline{3}}=ff^{\prime },  \notag \\
^{C}\Gamma _{14}^{\overline{4}} &=&\frac{u^{1}(YY^{\prime \prime }-(Y^{\prime
})^{2})}{Y^{2}},  \notag \\
\quad ^{C}\Gamma _{34}^{\overline{4}} &=&\frac{u^{3}(ff^{\prime \prime }-(f^{\prime
})^{2})}{f^{2}},  \notag \\
^{C}\Gamma _{\overline{1}4}^{\overline{4}} &=&\frac{Y^{\prime }}{Y},\ ^{C}\Gamma _{%
\overline{3}4}^{\overline{4}}=\frac{f^{\prime }}{f}.  \notag
\end{eqnarray}%
The Christoffel symbols of the metric $^{C}$ $\widehat{g}$ are occured when
the changes $X\rightarrow \widehat{X},\ Y\rightarrow \widehat{Y}\ $and $%
f\rightarrow \widehat{f}$ are done in (\ref{255}). Taking into account
equation (\ref{4}), $^{C}$ $\widehat{g}$ is harmonic with respect to $^{C}{}g
$ if and only if 
\begin{equation*}
\text{tr}(^{C}g^{-1}(^{C}\widehat{\Gamma }^{k}-^{C}\Gamma ^{k}))=0,\ k=1,\cdot
\cdot \cdot ,4,\bar{1},\cdot \cdot \cdot ,\bar{4}.
\end{equation*}%
Straightforward but long computations give us the harmonicity conditions as
follows:%
\begin{eqnarray*}
\frac{1}{X^{2}}(\widehat{X}^{\prime }\widehat{X}-X^{\prime }X)+\frac{1}{%
Y^{2}}(\widehat{Y}^{\prime }\widehat{Y}-Y^{\prime }Y)+\frac{1}{Y^{2}f^{2}}(%
\widehat{Y}^{\prime }\widehat{Y}\widehat{f}^{2}-Y^{\prime }Yf^{2}) &=&0, \\
-\frac{1}{Y^{2}f^{2}}(\widehat{f}\widehat{f}^{\prime }-ff^{\prime }) &=&0.
\end{eqnarray*}%
This completes the proof of the final theorem of the paper.
\end{proof}

\subsection{Data availability}

Data sharing not applicable to this article as no datasets were generated or
analysed during the current study.

\end{document}